\documentclass[11pt]{amsart}
\usepackage[margin=1in]{geometry}
\usepackage{amssymb}
\usepackage{amsthm}
\usepackage{amsmath}
\usepackage{mathrsfs}
\usepackage{amsbsy}
\usepackage[all]{xy}
\usepackage{bm}
\usepackage{hyperref}
\usepackage{tikz}
\usepackage{array}
\usepackage{float}
\usepackage{enumerate}
\usepackage{xcolor}
\usepackage{hhline}
\allowdisplaybreaks
\usepackage[noadjust]{cite}
\usepackage{caption}
\usepackage{tabu}
\usepackage{diagbox}

\newenvironment{enumerate*}%
  {\begin{enumerate}[(I)]%
    \setlength{\itemsep}{10pt}%
    \setlength{\parskip}{0pt}}%
  {\end{enumerate}}

\newtheorem{theorem}{Theorem}[section]
\newtheorem{proposition}[theorem]{Proposition}
\newtheorem{corollary}[theorem]{Corollary}
\newtheorem{conjecture}[theorem]{Conjecture}

\newtheorem{lemma}[theorem]{Lemma}

\theoremstyle{definition}

\DeclareMathOperator{\traj}{traj}
\DeclareMathOperator{\midp}{mid}
\DeclareMathOperator{\cyc}{cyc}
\DeclareMathOperator{\area}{area}
\DeclareMathOperator{\perim}{perim}
\DeclareMathOperator{\len}{length}

\newcommand{\dfn}[1]{\textcolor{blue}{\emph{#1}}}

\begin{document}

\title[]{Triangular-Grid Billiards and Plabic Graphs}

\author[Colin Defant and Pakawut Jiradilok]{Colin Defant}
\address[]{Department of Mathematics, Massachusetts Institute of Technology, Cambridge, MA 02139, USA}
\email{colindefant@gmail.com}
\author[]{Pakawut Jiradilok}
\address[]{Department of Mathematics, Massachusetts Institute of Technology, Cambridge, MA 02139, USA}
\email{pakawut@mit.edu}

\maketitle

\begin{abstract}
Given a polygon $P$ in the triangular grid, we obtain a permutation $\pi_P$ via a natural billiards system in which beams of light bounce around inside of $P$. The different cycles in $\pi_P$ correspond to the different trajectories of light beams. We prove that \[\area(P)\geq 6\cyc(P)-6\quad\text{and}\quad\perim(P)\geq\frac{7}{2}\cyc(P)-\frac{3}{2},\] where $\area(P)$ and $\perim(P)$ are the (appropriately normalized) area and perimeter of $P$, respectively, and $\cyc(P)$ is the number of cycles in $\pi_P$. The inequality concerning $\area(P)$ is tight, and we characterize the polygons $P$ satisfying $\area(P)=6\cyc(P)-6$. These results can be reformulated in the language of Postnikov's plabic graphs as follows. Let $G$ be a connected reduced plabic graph with essential dimension $2$. Suppose $G$ has $n$ marked boundary points and $v$ (internal) vertices, and let $c$ be the number of cycles in the trip permutation of $G$. Then we have \[v\geq 6c-6\quad\text{and}\quad n\geq\frac{7}{2}c-\frac{3}{2}.\] 
\end{abstract}

\section{Introduction}\label{SecIntro}

\subsection{Triangular-Grid Billiards}
Consider the infinite triangular grid in the plane, scaled so that each equilateral triangular grid cell has side length $1$ and oriented so that some of the grid lines are horizontal. We refer to the sides of these grid cells as \dfn{panes} because we will imagine that each pane either allows light to pass through it (like a window pane) or reflect off of it (like a mirror pane). Define a \dfn{grid polygon} to be a (not necessarily convex) polygon whose boundary is a union of panes. We assume that the boundary of a grid polygon (viewed as a closed curve) does not intersect itself. Suppose $P$ is a grid polygon whose boundary panes are $b_1,\ldots,b_n$, listed clockwise. Pick some boundary pane $b_{i}$, and emit a colored beam of light from the midpoint of $b_i$ into the interior of $P$ so that the light beam forms a $60^\circ$ angle with $b_i$ and travels either northeast, southeast, or west (depending on the orientation of $b_i$). The light beam will travel through the interior of $P$ until reaching the midpoint of a different boundary pane $b_{\pi(i)}$, which it will meet at a $60^\circ$ angle. This defines a permutation $\pi=\pi_P\colon[n]\to[n]$ (where $[n]:=\{1,\ldots,n\}$) called the \dfn{billiards permutation} of $P$. For example, if $P$ is the grid polygon in Figure~\ref{FigBill1}, then the cycle decomposition of $\pi_P$ is \[\pi_P=(1\,3\,32\,26\,6\,30\,2\,33\,25\,12\,14\,9\,21\,19\,29\,28\,4\,31)(5\,24\,13\,10\,20\,27)(7\,22\,23\,15\,17)(8\,11\,18\,16).\] 

\begin{figure}[ht]
  \begin{center}\includegraphics[height=6cm]{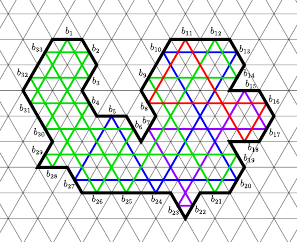}
  \end{center}
  \caption{A grid polygon $P$ with $33$ boundary panes. The billiards permutation $\pi_P$ has $4$ cycles. We have colored the $4$ different trajectories with different colors.}\label{FigBill1}
\end{figure}

One can interpret this definition of $\pi$ as a certain billiards process. Let us imagine that the boundary panes of $P$ are mirrors (and all other panes are transparent windows). When the light beam emitted from $b_i$ reaches $b_{\pi(i)}$, it will bounce off in such a way that the reflected beam forms a $60^\circ$ angle with $b_{\pi(i)}$. This reflected beam will then travel to $b_{\pi^2(i)}$, where it will bounce off at a $60^\circ$ angle and continue on to $b_{\pi^3(i)}$, and so on. We will be interested in the cycles of $\pi$. Given points $p$ and $p'$ in the plane, let us write $[p,p']$ for the line segment whose endpoints are $p$ and $p'$. Let $\midp(s)$ denote the midpoint of a line segment $s$. If $c=(i_1\,i_2\cdots i_r)$ is a cycle of $\pi$, then we define the \dfn{trajectory} of $c$ to be  \[\traj(c)=\bigcup_{j=1}^r[\midp(b_{i_j}),\midp(b_{\pi(i_j)})].\] The description of $\pi$ in terms of light beam billiards is convenient because we can imagine that the beams of light corresponding to different cycles have different colors; thus, we will use different colors to draw different trajectories (see Figure~\ref{FigBill1}). 

The investigation of billiards in planar regions is a classical and much-beloved topic in both dynamical systems and recreational mathematics \cite{Boldrighini, Croft, Croft2, DeTemple, DeTemple2, Gardner, Gutkin, Halpern, Sine}. However, the typical questions considered in previous works concern systems where the beams of light can have arbitrary initial positions and arbitrary initial directions. In contrast, our setup---which surprisingly appears to be new---imposes a great deal of rigidity by requiring each beam of light to start at the midpoint of a boundary pane and begin its journey in a direction that forms a $60^\circ$ angle with that boundary pane.  Although several traditional dynamically-flavored billiards problems (such as determining the existence of periodic trajectories) become trivial or meaningless under our rigid conditions, our setting affords some fascinating combinatorial/geometric questions. 

The major players in our story are the following quantities associated with a grid polygon $P$. The \dfn{perimeter} of $P$, denoted $\perim(P)$, is simply the number of boundary panes of $P$. We define the \dfn{area} of $P$, denoted $\area(P)$, to be the number of triangular grid cells in $P$.\footnote{Thus, our area measure is just the Euclidean area multiplied by the normalization factor $4/\sqrt 3$.} We write $\cyc(P)$ for the number of cycles in the associated permutation $\pi_P$, which is the same as the number of different light beam trajectories in the associated billiards system. Our main theorems address the following extremal question concerning the possible relationships between these quantities: How big must $\area(P)$ and $\perim(P)$ be in comparison with $\cyc(P)$?  

\begin{theorem}\label{thm:area}
If $P$ is a grid polygon, then \[\area(P)\geq 6\cyc(P)-6.\]
\end{theorem}

\begin{theorem}\label{thm:perim}
If $P$ is a grid polygon, then \[\perim(P)\geq \frac{7}{2}\cyc(P)-\frac{3}{2}.\]
\end{theorem}

The inequality in Theorem~\ref{thm:area} is tight, and we will characterize the grid polygons that achieve equality. Define a \dfn{unit hexagon} to be a grid polygon that is a regular hexagon of side length $1$. Let us construct a sequence of grid polygons $(P_k)_{k\geq 1}$ as follows. First, let $P_1$ be a unit hexagon. For $k\geq 2$, let $P_k=P_{k-1}\cup Q_{k}$, where $Q_{k}$ is a unit hexagon such that $P_{k-1}\cap Q_{k}$ is a single pane. We call a grid polygon $P_k$ obtained in this manner a \dfn{tree of unit hexagons}; see Figure~\ref{FigBill2} for an example with $k=9$. Since $\cyc(Q_k)=2$ for all $k\geq 2$, one can combine Corollary~\ref{lem:cut1} from below with an easy inductive argument to see that $\cyc(P_k)=k+1$ for all $k\geq 1$. Thus, $\area(P_k)=6k=6\cyc(P_k)-6$. 

\begin{figure}[ht]
  \begin{center}\includegraphics[height=4.86cm]{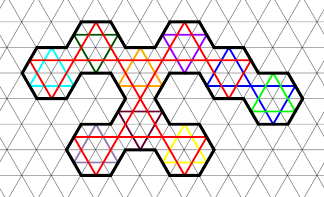}
  \end{center}
  \caption{A tree of unit hexagons $P_9$ with $\cyc(P_9)=10$, $\area(P_9)=54$, and $\perim(P_9)=38$.}\label{FigBill2}
\end{figure}

\begin{theorem}\label{thm:equality}
If $P$ is a grid polygon, then $\area(P)=6\cyc(P)-6$ if and only if $P$ is a tree of unit hexagons. 
\end{theorem}

On the other hand, we believe that Theorem~\ref{thm:perim} is not tight. After drawing several examples of grid polygons, we have arrived at the following conjecture. 

\begin{conjecture}\label{conj:4n-2}
If $P$ is a grid polygon, then \[\perim(P)\geq 4\cyc(P)-2.\]
\end{conjecture}

If Conjecture~\ref{conj:4n-2} is true, then it is tight. Indeed, if $P_k$ is a tree of unit hexagons as described above, then $\perim(P_k)=4k+2=4\cyc(P_k)-2$. 

Of fundamental importance in our analysis of the billiards system of a grid polygon $P$ are the \dfn{triangular trajectories} of $P$, which are just the trajectories of the $3$-cycles in $\pi_P$. One of the crucial ingredients in the proofs of Theorems~\ref{thm:area}--\ref{thm:equality} is the following result, which we deem to be noteworthy on its own.  

\begin{theorem}\label{thm:HypothesisA}
Let $P$ be a grid polygon, and let $c$ be a cycle of size $m$ in $\pi_P$. Then the trajectory $\traj(c)$ intersects at most $m-2$ triangular trajectories of $P$ (excluding $\traj(c)$ itself if $m=3$).
\end{theorem}

\subsection{Plabic Graphs}\label{subsec:plabic} 

A \dfn{plabic graph} is a planar graph $G$ embedded in a disc such that each vertex is colored either black or white. We assume that the boundary of the disc has $n$ marked points labeled clockwise as $b_1^*,\ldots,b_n^*$ so that each $b_i^*$ is connected via an edge to exactly one vertex of $G$. Following \cite{LamPostnikov}, we will also assume that every vertex of $G$ is incident to exactly $3$ edges, including edges connected to the boundary of the disc (the study of general plabic graphs can be reduced to this case). In his seminal article \cite{Postnikov}, Postnikov introduced plabic graphs---along with several other families of combinatorial objects---in order to parameterize cells in the totally nonnegative Grassmannian. These graphs have now found remarkable applications in a variety of fields such as cluster algebras, knot theory, polyhedral geometry, scattering amplitudes, and shallow water waves \cite{Arkani1, Fomin, Galashin, Kodama, LamPostnikov, Lukowski, Parisi, Shende}. 

Imagine starting at a marked boundary point $b_i^*$ and traveling along the unique edge connected to $b_i^*$. Each time we reach a vertex, we follow the \emph{rules of the road} by turning right if the vertex is black and turning left if the vertex is white. Eventually, we will reach a marked boundary point $b_{\pi(i)}^*$. The path traveled is called the \dfn{trip} starting at $b_i^*$. Considering the trips starting at all of the different marked boundary points yields a permutation $\pi=\pi_G:[n]\to[n]$ called the \dfn{trip permutation} of $G$. We say $G$ is \dfn{reduced} if it has the minimum number of faces among all plabic graphs with the same trip permutation. Figure~\ref{FigBill3} shows a reduced plabic graph $G$ whose trip permutation is the cycle $\pi_G=(1\,3\,5\,2\,4)$. 

\begin{figure}[ht]
  \begin{center}\includegraphics[height=5.2cm]{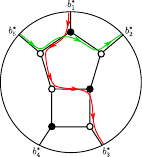}
  \end{center}
  \caption{A reduced plabic graph whose trip permutation is the cycle $(1\,3\,5\,2\,4)$. The fact that $\pi_G(1)=3$ is illustrated by the red trip starting at $b_1^*$ and ending at $b_3^*$. Similarly, $\pi_G(5)=2$ because the green trip starting at $b_5^*$ ends at $b_2^*$.}\label{FigBill3}
\end{figure}

Given a grid polygon $P$, one can obtain a reduced plabic graph $G(P)$ via a planar dual construction. Let us say an equilateral triangle with a horizontal side is \dfn{right-side up} (respectively, \dfn{upside down}) if its horizontal side is on its bottom (respectively, top). We refer to this property of a triangle (right-side up or upside down) as its \dfn{orientation}. Place a black (respectively, white) vertex at the center of each right-side up (respectively, upside down) equilateral triangular grid cell inside of $P$. Whenever two such grid cells share a side, draw an edge between the corresponding vertices. Finally, encompass $P$ in a disc, draw a marked point $b_i^*$ on the boundary of the disc corresponding to each boundary pane $b_i$ of $P$, and draw an edge from $b_i^*$ to the vertex drawn inside the unique grid cell that has $b_i$ as a side. See Figure~\ref{FigBill4}. 

\begin{figure}[ht]
  \begin{center}\includegraphics[height=5.416cm]{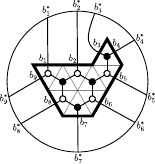}\qquad \includegraphics[height=5.416cm]{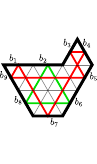}\qquad\includegraphics[height=5.416cm]{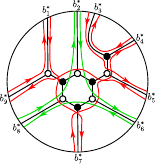}
  \end{center}
  \caption{On the left is a grid polygon $P$ overlaid with the corresponding plabic graph $G(P)$. The middle image shows the trajectories in the billiards system of $P$ to illustrate that its billiards permutation is $\pi_{P}=(1\,7\,4\,3\,5\,9)(2\,6\,8)$. The right image shows the trips of $G(P)$ to illustrate that its trip permutation is $\pi_{G(P)}=(1\,7\,4\,3\,5\,9)(2\,6\,8)$.  }\label{FigBill4}
\end{figure}

It is immediate from the relevant definitions that the trip permutation $\pi_{G(P)}$ is equal to the billiards permutation $\pi_P$. For example, if $P$ and $G(P)$ are as in Figure~\ref{FigBill4}, then $\pi_P=\pi_{G(P)}$ is the permutation with cycle decomposition $(1\,7\,4\,3\,5\,9)(2\,6\,8)$. 

\subsection{Membranes}

In the recent paper \cite{LamPostnikov}, Lam and Postnikov introduced \emph{membranes}, which are certain triangulated $2$-dimensional surfaces embedded in a Euclidean space. The definition of a membrane relies on a choice of an irreducible root system, and most of the discussion in \cite{LamPostnikov} centers around membranes of type $A$. They discussed how type $A$ membranes are in a sense dual to plabic graphs, and they further related type $A$ membranes to the theory of cluster algebras. A membrane is \dfn{minimal} if it has the minimum possible surface area among all membranes with the same boundary; Lam and Postnikov showed how to associate a reduced plabic graph $G(M)$ to each minimal type $A$ membrane $M$. They then defined the \dfn{essential dimension} of a reduced plabic graph $G_0$ to be the smallest positive integer $d$ such that there exists a minimal membrane $M$ of type~$A_d$ with $G(M)=G_0$. They proved that if $G_0$ has $n$ marked boundary points, then the essential dimension of $G_0$ is at most $n-1$, with equality holding if and only if there exists $k\in[n-1]$ such that $\pi_{G_0}(i)=i+k\pmod{n}$ for all $i\in[n]$ (this is equivalent to saying that $G_0$ corresponds to the top cell in the totally nonnegative Grassmannian $\text{Gr}_{k,n}^{\geq 0}$). Other than this result, there is essentially nothing known about essential dimensions of plabic graphs. Our original motivation for this project was to initiate the investigation of essential dimensions by studying plabic graphs of essential dimension $2$ in detail. 

Consider the class of triangulated surfaces in the triangular grid that can be obtained by iteratively wedging grid polygons. In other words, $Q$ is in this class if there are grid polygons $P_1,\ldots,P_k$ such that $P_{i+1}\cap(P_1\cup\cdots\cup P_i)$ is a single point for all $i\in[k-1]$ and such that $Q=P_1\cup\cdots \cup P_k$. In this case, we call the grid polygons $P_1,\ldots,P_k$ the \dfn{components} of $Q$. See Figure~\ref{FigBill5}. As mentioned in \cite{LamPostnikov}, the class we have just described is the same as the class of membranes of type $A_2$.  Such a membrane $M$ is automatically minimal (since it is determined by its boundary). In order to understand these membranes and their associated reduced plabic graphs, it suffices to understand grid polygons and their associated reduced plabic graphs. Indeed, the reduced plabic graphs associated to the components of $M$ are basically the same as the connected components of the reduced plabic graph associated to $M$; thus, restricting our focus to grid polygons is the same as restricting our focus to connected plabic graphs. Furthermore, if $M=P$ is a grid polygon, then the definition that Lam and Postnikov gave for the reduced plabic graph $G(M)$ associated to $M$ (viewed as a membrane) is exactly the same as the definition that we gave in Section~\ref{subsec:plabic} for the reduced plabic graph $G(P)$ associated to $P$ (viewed as a grid polygon). In other words, understanding plabic graphs of essential dimension $2$ and their trip permutations is equivalent to understanding grid polygons and their billiards permutations.

\begin{figure}[ht]
  \begin{center}\includegraphics[height=3.32cm]{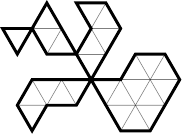}
  \end{center}
  \caption{A membrane of type $A_2$ with $5$ components.}\label{FigBill5}
\end{figure}

As a consequence of the preceding discussion, we can reformulate Theorems~\ref{thm:area} and~\ref{thm:perim} in the language of plabic graphs. 

\begin{corollary}\label{cor:plabic}
Let $G$ be a connected reduced plabic graph with essential dimension $2$. Suppose $G$ has $n$ marked boundary points and $v$ vertices, and let $c$ be the number of cycles in the trip permutation $\pi_G$. Then \[v\geq 6c-6\quad\text{and}\quad n\geq\frac{7}{2}c-\frac{3}{2}.\] 
\end{corollary}

\begin{proof}
By the preceding discussion, there is a grid polygon $P$ such that $G=G(P)$. We have $\perim(P)=n$, $\area(P)=v$, and $\cyc(P)=c$, so the corollary follows from Theorems~\ref{thm:area} and~\ref{thm:perim}.
\end{proof}

\subsection{Outline}
Section~\ref{sec:triangles} is devoted to the proof of Theorem~\ref{thm:HypothesisA}. In Section~\ref{sec:area}, we apply Theorem~\ref{thm:HypothesisA} to prove Theorems~\ref{thm:area}, \ref{thm:perim}, and~\ref{thm:equality}. We believe that our work opens the door to several new combinatorial and geometric questions; we have collected many ideas for future work in Section~\ref{sec:conclusion}.  

\section{Triangles Intersecting a Trajectory}\label{sec:triangles}

Our goal in this section is to prove Theorem~\ref{thm:HypothesisA}. We begin with a lemma that establishes this theorem in the special case when $m=3$. 

\begin{lemma}\label{lem:triangle}
Suppose $\Delta$ is a triangular trajectory in the billiards system of a grid polygon $P$. There is at most one triangular trajectory $\Delta'$ in $P$ that intersects $\Delta$ and is not equal to $\Delta$. If such a trajectory $\Delta'$ exists, then its orientation must be opposite to that of $\Delta$.  
\end{lemma}

\begin{proof}
If two triangular trajectories intersect, then neither one can have a vertex in the interior of the region bounded by the other. This forces the two triangular trajectories to have opposite orientations. It also implies that every side of the first trajectory intersects the second trajectory and vice versa (i.e., the trajectories intersect in $6$ points). It follows from these observations that a triangular trajectory cannot intersect two other triangular trajectories.
\end{proof}

Let us fix some additional notation and terminology concerning trajectories. When we refer to a \emph{line segment}, we assume by default that it contains its endpoints and that it is not a single point. Let $P$ be a grid polygon, and let $c$ be an $m$-cycle in $\pi_P$. Since Lemma~\ref{lem:triangle} tells us that Theorem~\ref{thm:HypothesisA} is true when $m=3$, we will assume that $m\geq 4$. Let $z_1,\ldots,z_m$ be the points where the trajectory $\traj(c)$ intersects the boundary of $P$, listed clockwise around the boundary. For convenience, let $z_{m+1}=z_1$. Imagine traversing the boundary of $P$ clockwise, and let $B_i$ be the part of the boundary traversed between $z_i$ and $z_{i+1}$, including $z_i$ and $z_{i+1}$. We call each $B_i$ a \dfn{shoreline} of $\traj(c)$. Note that $\traj(c)$ is the union of $m$ line segments, each of which has its endpoints in $\{z_1,\ldots,z_m\}$. Let $C$ be the set of points where two of these line segments intersect each other (including $z_1,\ldots,z_m$). If we ``cut'' $\traj(c)$ at each point in $C$, we will break each of the $m$ line segments into smaller line segments that we call the \emph{fragments} of $\traj(c)$. More precisely, we say a line segment $f\subseteq\traj(c)$ is a \dfn{fragment} of $\traj(c)$ if the endpoints of $f$ belong to $C$ and if the relative interior of $f$ does not contain any points from $C$. We say a fragment $f$ \dfn{sees} a shoreline $B$ if there exist a point $p$ in the relative interior of $f$ and a point $p'$ in $B$ that is not an endpoint of $B$ such that the relative interior of the line segment $[p,p']$ lies inside of $P$ and does not contain any points from $\traj(c)$. If $f$ is a fragment of $\traj(c)$ that sees the shoreline $B$, then we write $E(f,B)$ for the equilateral triangle that has $f$ as one of its sides and that lies on the side of $f$ opposite to $B$. See Figure~\ref{FigBill6}. 

\begin{figure}[ht]
  \begin{center}\includegraphics[height=3.7cm]{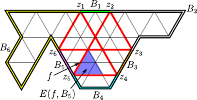}
  \end{center}
  \caption{In red is the trajectory of a $6$-cycle in the billiards system of a grid polygon. This trajectory has $12$ fragments. The $6$ shorelines are represented with different colors. Each of $B_1,B_3,B_5$ is seen by exactly $2$ of the fragments, while each of $B_2,B_4,B_6$ is seen by exactly $1$ of the fragments. We have labeled one of the fragments $f$ that sees $B_5$, and we have shaded the triangle $E(f,B_5)$ in blue. }\label{FigBill6}
\end{figure}

Fix a shoreline $B$, and let $f_1,\ldots, f_k$ be the fragments of $\traj(c)$ that see $B$. Then $f_1\cup\cdots\cup f_k\cup B$ is a piecewise-linear curve that bounds a polygonal region $R$. Let us assume that $f_1,\ldots,f_k$ are listed in clockwise order around $R$ so that $f_1$ and $f_k$ touch $B$. For each $j\in[k-1]$, let $\theta_j$ be the interior angle of $R$ at the point of intersection of $f_j$ and $f_{j+1}$. It is straightforward to see that $\theta_j$ is either $60^\circ$ or $120^\circ$. Let $K=K(B)=1+\sum_{j=1}^{k-1} \beta_j$, where $\beta_j=2$ if $\theta_j=60^\circ$ and $\beta_j=1$ if $\theta_j=120^\circ$. A schematic illustration of this situation is shown in Figure~\ref{FigBill7}. In that figure, we have  $\theta_1=\theta_3=120^\circ$ and $\theta_2=60^\circ$, so $\beta_1=\beta_3=1$, $\beta_2=2$, and $K(B)=1+(1+2+1)=5$. 

Imagine standing at the point $B\cap f_k$ and facing toward $B$. Walk along the shoreline $B$ to reach $B\cap f_1$; you should now be facing toward the shoreline that comes immediately after $B$ in clockwise order. The net change in your direction during this walk is $(180-60K)^\circ$ clockwise.\footnote{A net change of, say, $120^\circ$ clockwise is different from a net change of $-240^\circ$ clockwise. In the former case, you spun $120^\circ$ clockwise; in the latter case, you spun $240^\circ$ counterclockwise.} To see this, consider instead walking from $B\cap f_k$ to $B\cap f_1$ by traversing the fragments $f_k,\ldots,f_1$; this will result in the same net change in direction. You first turn $60^\circ$ clockwise to get onto the fragment $f_k$. Whenever you transfer from $f_{j+1}$ to $f_j$, you turn $(180^\circ-\theta_j)=(60\beta_j)^\circ$ counterclockwise, which is the same as $(-60\beta_j)^\circ$ clockwise. At the end, you turn another $60^\circ$ clockwise to get off of $f_1$ and face toward the next shoreline. Overall, your net change in direction is $60^\circ+\sum_{j=1}^{k-1}(-60\beta_j)^\circ +60^\circ=(180-60K)^\circ$ clockwise. In the example shown in Figure~\ref{FigBill7}, the net change of direction would be $60^\circ-60^\circ-120^\circ-60^\circ+60^\circ=-120^\circ=(180-60K)^\circ$ clockwise (i.e., $120^\circ$ counterclockwise).  

We are going to prove that the number of triangular trajectories that intersect $\traj(c)$ and touch $B$ is at most $K$. First, we need the following lemmas.

\begin{lemma}\label{lem:E(f,B)}
Preserve the notation from above. For each $i\in[k]$, the boundary of the grid polygon $P$ does not intersect the interior of $E(f_i,B)$. 
\end{lemma}

\begin{proof}
Without loss of generality, we may assume that $E(f_i,B)$ is right-side up and has $f_i$ as its (horizontal) bottom side. Let $\mathcal{W}$ denote the intersection of the boundary of $P$ and the interior of $E(f_i, B)$. Suppose for the sake of contradiction that $\mathcal{W}$ is nonempty. It is not hard to see that there is a point $w^{\ast} \in \mathcal{W}$ whose distance to $f_i$ is the minimum among all points in $\mathcal{W}$. Let $T_{w^*}$ be the unique equilateral triangle that contains $w^*$ as a vertex and has one of its sides contained in $f_i$. By the minimality of the distance from $w^{\ast}$ to $f_i$, we observe that $T_{w^{\ast}}$ does not contain other points from the boundary of $P$ besides $w^{\ast}$. See Figure~\ref{FigBill13}.

Note that the space $P\setminus(R\cup T_{w^*})$ has two connected components: a left region whose closure contains the left endpoint of $f_i$ and a right region whose closure contains the right endpoint of $f_i$. Imagine following the trajectory $\traj(c)$ starting at the right endpoint of $f_i$ and continuing through the left endpoint of $f_i$. You will land in the left region. If you continue following the trajectory, you will eventually come back to the right endpoint of $f_i$, which is in the right region. Since $\traj(c)$ does not intersect the interior of $R$, it must travel from the left region to the right region through $T_{w^{\ast}}$. However, this means that there exists a horizontal boundary pane of $P$ inside $T_{w^{\ast}}$, which is a contradiction.
\end{proof}

\begin{figure}[ht]
\begin{center}\includegraphics[height=6.552cm]{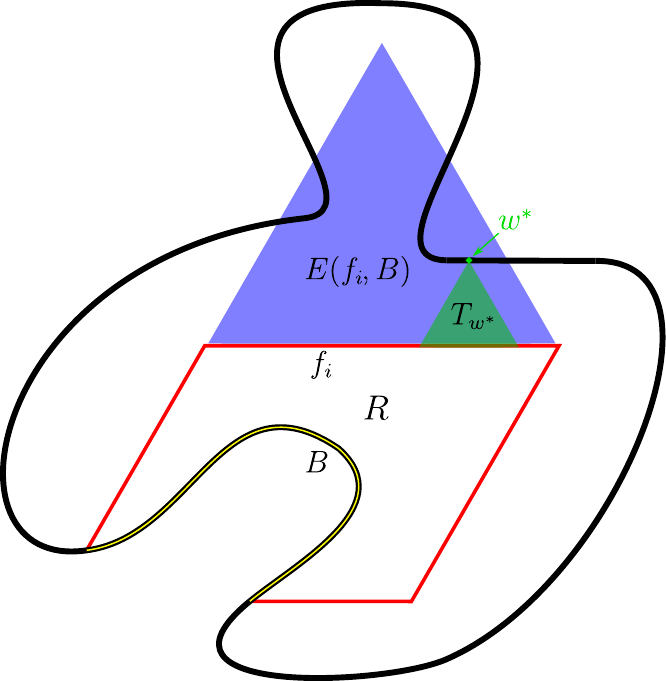}
  \end{center}
  \caption{In this schematic drawing, the boundary of $P$ passes through the interior of $E(f_i,B)$, where $B$ is the shoreline indicated by a thin yellow strip and $f_i$ is a fragment that sees $B$. The proof of Lemma~\ref{lem:E(f,B)} derives a contradiction from this setup. }\label{FigBill13}
\end{figure}

\begin{lemma}\label{lem:AtMost1Intersection}
Preserve the notation from above. If a triangular trajectory intersects $\traj(c)$ and touches $B$, then it touches $B$ at exactly $1$ point.  
\end{lemma}

\begin{proof}
Let $\Delta$ be a triangular trajectory that intersects $\traj(c)$ and touches $B$. Let $f$ be a fragment of $\traj(c)$ that intersects $\Delta$. Let $a_1,a_2,a_3$ be the vertices of $\Delta$. It is easy to see that $\Delta$ cannot have all $3$ of its vertices on $B$. Now suppose $\Delta$ has exactly $2$ of its vertices, say $a_1$ and $a_2$, on $B$. By rotating $P$ if necessary, we may assume $[a_1,a_2]$ is a horizontal line segment. The boundary of $P$ does not pass through the interior of the region bounded by $\Delta$; combining this with the observation that $\traj(c)$ does not intersect $[a_1,a_2]$; we find that every fragment of $\traj(c)$ passing through the interior of the region bounded by $\Delta$ must be horizontal. It follows that $f$ is a horizontal line segment that intersects $[a_1,a_3]$ and $[a_2,a_3]$. 
However, this forces $a_3$ to be in the interior of $E(f,B)$, contradicting Lemma~\ref{lem:E(f,B)}.     
\end{proof}

\begin{lemma}\label{lem:K_iThing}
Preserve the notation from above. There are at most $K(B)$ triangular trajectories in the billiards system of $P$ that intersect $\traj(c)$ and touch $B$. 
\end{lemma}

\begin{proof}
We illustrate the proof in Figure~\ref{FigBill7}. For each $i\in[k]$ and each side $s$ of the triangle $E(f_i,B)$, it follows from Lemma~\ref{lem:E(f,B)} that there is a unique line segment $L(s)$ containing $s$ that does not pass through the exterior of $P$ and whose endpoints are on the boundary of $P$. Let \[X=\bigcup_{i=1}^k\bigcup_{\substack{s\text{ a side}\\ \text{of }E(f_i,B)}}L(s).\] Define an \dfn{$X$-region} to be the closure of a connected component of $P\setminus X$; let $\mathcal R_X$ be the set of $X$-regions. Say an $X$-region $V$ is \dfn{hospitable} if it contains at least one side of at least one of the triangles $E(f_1,B),\ldots,E(f_k,B)$; otherwise, say $V$ is \dfn{inhospitable}. Note that $E(f_1,B),\ldots,E(f_k,B)$ are $k$ different hospitable $X$-regions. It is straightforward to see that an $X$-region is of the form $E(f_i,B)$ for some $i\in[k]$ if and only if it does not contain a line segment in the boundary of $P$. Let $R$ be the $X$-region whose boundary is $f_1\cup\cdots\cup f_k\cup B$. Let $R'$ (respectively, $R''$) be the unique $X$-region other than $R$ that contains the point $f_1\cap B$ (respectively, $f_k\cap B$). Note that $R$, $R'$, and $R''$ are hospitable. One can readily check that there are exactly $K(B)-1$ hospitable $X$-regions in $\mathcal R_X\setminus\{E(f_1,B),\ldots,E(f_k,B),R,R',R''\}$; let $U_1,\ldots,U_{K(B)-1}$ be these $X$-regions. Finally, let $I_1,\ldots,I_\ell$ be the inhospitable $X$-regions. 

Let $t$ be the number of triangular trajectories that intersect $\traj(c)$ and touch $B$. Let $p_1,\ldots,p_{3t}$ be the points where these triangular trajectories touch the boundary of $P$. For each $X$-region $V$, let $g(V)=|V\cap\{p_1,\ldots,p_{3t}\}|$. Each of the points $p_1,\ldots,p_{3t}$ belongs to exactly one $X$-region, so $\sum_{V\in\mathcal R_X}g(V)=3t$. It follows from Lemma~\ref{lem:AtMost1Intersection} that $g(R)=t$. Lemma~\ref{lem:E(f,B)} immediately implies that $g(E(f_i,B))=0$ for all $i\in[k]$. Furthermore, using Lemma~\ref{lem:E(f,B)}, one can readily check that $g(R')\leq 1$, $g(R'')\leq 1$, $g(U_j)\leq 2$ for all $j\in[K(B)-1]$, and $g(I_h)=0$ for all $h\in[\ell]$. Thus, \[3t=\sum_{V\in\mathcal R_X}g(V)\leq t+1+1+\sum_{i=1}^k0+\sum_{j=1}^{K(B)-1}2+\sum_{h=1}^\ell 0=t+2K(B).\]  Hence, $t\leq K(B)$. 
\end{proof}

\begin{figure}[ht]
  \begin{center}\includegraphics[height=8cm]{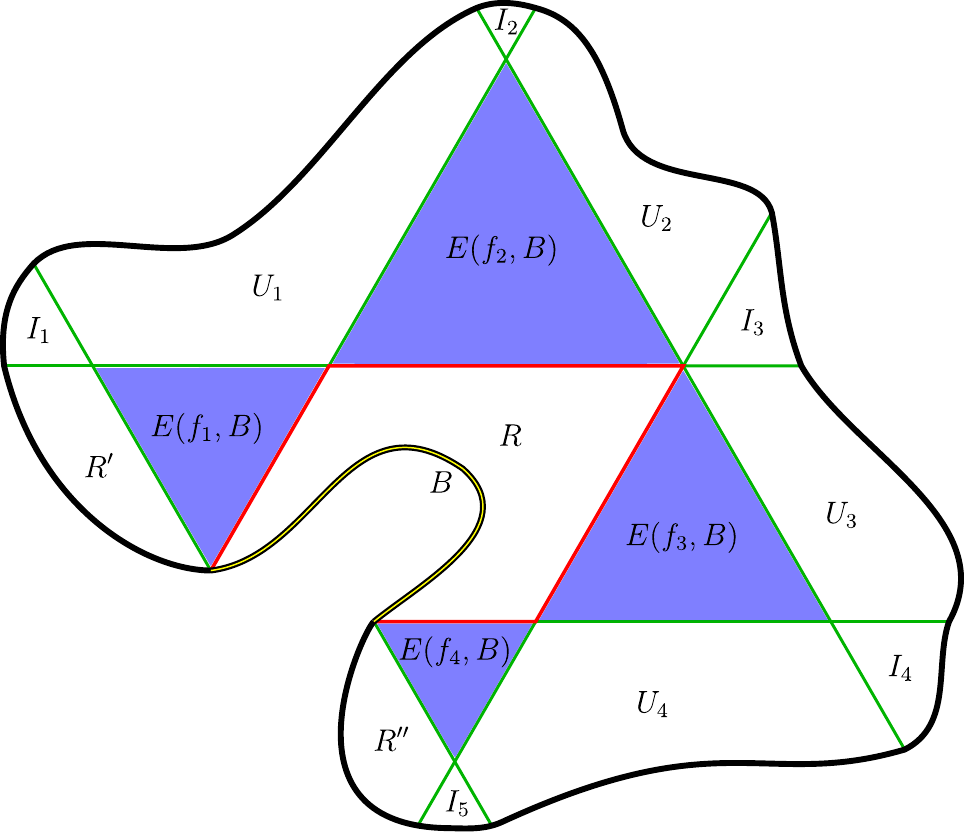}
  \end{center}
  \caption{A schematic illustration of the proof of Lemma~\ref{lem:K_iThing}. The curvy black curve is meant to represent the boundary of $P$. The red line segments are the fragments $f_1,f_2,f_3,f_4$, and the shoreline $B$ is marked with a thin yellow strip. The set $X$ is the union of the red and green line segments. In this example, $K(B)=5$. }\label{FigBill7}
\end{figure}

We are now in a position to complete the proof of Theorem~\ref{thm:HypothesisA}. 

\begin{proof}[Proof of Theorem~\ref{thm:HypothesisA}]
Let $P$ be a grid polygon, and let $c$ be an $m$-cycle in $\pi_P$. If $m=3$, then Theorem~\ref{thm:HypothesisA} follows from Lemma~\ref{lem:triangle}, so we may assume $m\geq 4$. Let $B_1,\ldots, B_m$ be the shorelines of $\traj(c)$. For each shoreline $B_i$, we define the integer $K(B_i)$ as above. Lemma~\ref{lem:K_iThing} tells us that there are at most $K(B_i)$ triangular trajectories that intersect $\traj(c)$ and touch $B_i$, and Lemma~\ref{lem:AtMost1Intersection} tells us that each such triangular trajectory touches $B_i$ in exactly $1$ point. Therefore, the total number of points where the triangular trajectories that intersect $\traj(c)$ touch the boundary of $P$ is at most $\sum_{i=1}^mK(B_i)$. Since each triangular trajectory touches the boundary of $P$ in exactly $3$ points, we deduce that the total number of triangular trajectories that intersect $\traj(c)$ is at most $\frac 13\sum_{i=1}^mK(B_i)$. Hence, the proof will be complete if we can show that $\frac 13\sum_{i=1}^mK(B_i)=m-2$. 

Preserve the notation from above. Imagine traversing the boundary of $P$ clockwise, starting and ending at $z_1$. We saw in our discussion above that the net change in your direction when you traverse the shoreline $B_i$ is $(180-60K(B_i))^\circ$ clockwise. Thus, the net change in your direction when you traverse the entire boundary of $P$ is $\sum_{i=1}^m(180-60K(B_i))^\circ$ clockwise. But this net change must be $360^\circ$, so  $\sum_{i=1}^m(180-60K(B_i))=360$. Manipulating this equation yields $\frac 13\sum_{i=1}^mK(B_i)=m-2$, as desired. 
\end{proof}

\section{Areas and Perimeters}\label{sec:area}

We will find it useful to break grid polygons into smaller grid polygons; the following lemma allows us to understand the effect that this has on the enumeration of the cycles in the associated billiards systems. 

\begin{lemma}\label{lem:general_cut}
Let $P$ be a grid polygon, and suppose $P=P_1\cup P_2$, where $P_1$ and $P_2$ are grid polygons such that $P_1\cap P_2$ is a union of $\eta$ different panes. Let $\delta_i$ be the number of different trajectories in the billiards system of $P_i$ that touch $P_1\cap P_2$. Then $\cyc(P)\leq\cyc(P_1)+\cyc(P_2)-\delta_1-\delta_2+\eta$. 
\end{lemma}

\begin{proof}
For each $i\in\{1,2\}$, the billiards system of $P_i$ contains $\cyc(P_i)-\delta_i$ trajectories that do not touch $P_1\cap P_2$, and these are also trajectories in the billiards system of $P$. It is straightforward to see that the billiards system of $P$ has at most $\eta$ trajectories that intersect $P_1\cap P_2$. 
\end{proof}

If $\eta=1$ in the preceding lemma, then $\delta_1=\delta_2=1$, and there must be exactly one trajectory in the billiards system of $P$ that intersects $P_1\cap P_2$. Hence, we have the following useful corollary.  

\begin{corollary}\label{lem:cut1}
Let $P=P_1\cup P_2$, where $P_1$ and $P_2$ are grid polygons such that $P_1\cap P_2$ is a single pane. Then $\cyc(P)=\cyc(P_1)+\cyc(P_2)-1$.
\end{corollary}

Let us say a grid polygon $P$ is \dfn{primitive} if there do not exist grid polygons $P_1$ and $P_2$ such that $P=P_1\cup P_2$ and such that $P_1\cap P_2$ is a single pane. Corollary~\ref{lem:cut1} will allow us to restrict our attention to primitive grid polygons. We will often need to handle the grid polygons in Figure~\ref{FigBill10} (and their rotations) separately. The proof of the next lemma is the main place where we apply Theorem~\ref{thm:HypothesisA}, which we proved in Section~\ref{sec:triangles}. 

\begin{figure}[ht]
  \begin{center}\includegraphics[height=1.979cm]{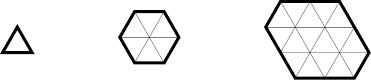}
  \end{center}
  \caption{Three primitive grid polygons.}\label{FigBill10}
\end{figure}

\begin{lemma}\label{lem:alpha_3_bound}
Let $P$ be a primitive grid polygon that is not a rotation of one of the grid polygons in Figure~\ref{FigBill10}. Let $\alpha_m$ be the number of $m$-cycles in $\pi_P$. Then $\alpha_3\leq\alpha_4+\sum_{m\geq 5}(m-2)\alpha_m$. 
\end{lemma}

\begin{proof}
For each $3$-cycle $c$ in $\pi_P$, let $F(c)$ be the set of cycles in $\pi_P$ whose trajectories intersect $\traj(c)$. Let $Y$ be the set of $3$-cycles $c$ in $\pi_P$ such that $F(c)$ contains only $3$-cycles and $4$-cycles. Using the hypothesis that $P$ is primitive and not a rotation of one of the polygons in Figure~\ref{FigBill10}, it is straightforward (though somewhat tedious) to verify if $c\in Y$, then $F(c)$ contains at least two $4$-cycles. On the other hand, Theorem~\ref{thm:HypothesisA} tells us that if $d$ is a $4$-cycle in $\pi_P$, then there are at most two $3$-cycles $c$ in $Y$ with $d\in F(c)$. Therefore, $|Y|\leq\alpha_4$. 

If $c$ is a $3$-cycle in $\pi_P$ that is not in $Y$, then $F(c)$ contains at least one cycle of size at least $5$. Theorem~\ref{thm:HypothesisA} tells us that if $d$ is an  $m$-cycle in $\pi_P$, then there are at most $m-2$ different $3$-cycles $c$ such that $d\in F(c)$. This implies that $\alpha_3-|Y|\leq\sum_{m\geq 5}(m-2)\alpha_m$. 
\end{proof}

We can now prove our main theorem concerning perimeters of grid polygons.

\begin{proof}[Proof of Theorem~\ref{thm:perim}]
The result is trivial if $\area(P)=1$, so we may assume $\area(P)\geq 2$ and proceed by induction on $\area(P)$. If $P$ is not primitive, then we can write $P=P_1\cup P_2$, where $P_1$ and $P_2$ are smaller grid polygons such that $P_1\cap P_2$ is a single pane. By induction, we have $\perim(P_i)\geq\frac{7}{2}\cyc(P_i)-\frac 32$ for each $i\in\{1,2\}$. Combining this with Corollary~\ref{lem:cut1} yields \[\perim(P)=\perim(P_1)+\perim(P_2)-2\geq \frac{7}{2}(\cyc(P_1)+\cyc(P_2)-1)-\frac{3}{2}=\frac{7}{2}\cyc(P)-\frac 32,\] as desired. Hence, we may assume $P$ is primitive. We easily verify that the desired inequality holds if $P$ is a rotation of one of the grid polygons in Figure~\ref{FigBill10}; hence, let us assume that this is not the case. Let $\alpha_m$ be the number of $m$-cycles in $\pi_P$. We have \[\frac{\perim(P)}{\cyc(P)}=\frac{3\alpha_3+\sum_{m\geq 4}m\alpha_m}{\alpha_3+\sum_{m\geq 4}\alpha_m}.\] For fixed values of $\alpha_4,\alpha_5,\ldots$, the function $x\mapsto\displaystyle \frac{3x+\sum_{m\geq 4}m\alpha_m}{x+\sum_{m\geq 4}\alpha_m}$ is decreasing in $x$ whenever $x\geq 0$. Therefore, we can apply Lemma~\ref{lem:alpha_3_bound} to find that \[\frac{\perim(P)}{\cyc(P)}\geq\frac{3(\alpha_4+\sum_{m\geq 5}(m-2)\alpha_m)+\sum_{m\geq 4}m\alpha_m}{(\alpha_4+\sum_{m\geq 5}(m-2)\alpha_m)+\sum_{m\geq 4}\alpha_m}=\frac{7\alpha_4+\sum_{m\geq 5}(4m-6)\alpha_m}{2\alpha_4+\sum_{m\geq 5}(m-1)\alpha_m}.\] Since $4m-6\geq\frac{7}{2}(m-1)$ for all $m\geq 5$, we conclude that $\perim(P)\geq\frac{7}{2}\cyc(P)>\frac{7}{2}\cyc(P)-\frac{3}{2}$. 
\end{proof}

We now proceed to prove Theorems~\ref{thm:area} and~\ref{thm:equality}. Recall that we have scaled the triangular grid so that each pane has length $1$. Given a cycle $c$ in $\pi_P$, we write $\len(\traj(c))$ for the \dfn{length} of $\traj(c)$, which is just the total length of the line segments in $\traj(c)$. It is easy to check that $\area(P)=\frac{2}{3}\sum_{c}\len(\traj(c))$, where the sum is over all cycles in $\pi_P$. 

\begin{proposition}\label{prop:indecomposable}
If $P$ is a primitive grid polygon that is not a rotation of one of the grid polygons in Figure~\ref{FigBill10}, then $\area(P)\geq 6\cyc(P)$. 
\end{proposition}

\begin{proof}
The proof is by induction on $\area(P)$. Because $P$ is primitive, all of the triangular trajectories in the billiards system of $P$ have length at least $15/2$. Let us first assume that every triangular trajectory in the billiards system of $P$ has length strictly greater than $15/2$. This forces every triangular trajectory to have length at least $21/2$. Because $P$ is primitive, every line segment in a trajectory in the billiards system of $P$ has length at least $3/2$. Therefore, $\len(\traj(c))\geq 3m/2$ for every $m$-cycle $c$ in $\pi_P$. It is also straightforward to check that every $4$-cycle in $\pi_P$ has a trajectory of length at least $9$ and that every $5$-cycle in $\pi_P$ has a trajectory of length at least $21/2$. Let $\alpha_m$ denote the number of $m$-cycles in $\pi_P$. Because $\area(P)=\frac{2}{3}\sum_{c}\len(\traj(c))$ (with the sum ranging over all cycles in $\pi_P$), we have \[\area(P)\geq \frac{2}{3}\left((21/2)\alpha_3+9\alpha_4+(21/2)\alpha_5+\sum_{m\geq 6}(3m/2)\alpha_m\right)=7\alpha_3+6\alpha_4+7\alpha_5+\sum_{m\geq 6}m\alpha_m\] \[\geq 6\cyc(P).\] 

Now assume that the billiards system of $P$ contains at least one triangular trajectory $\Delta$ of length $15/2$. Then up to rotation, $P$ must have the following shape, where curvy curves are schematic illustrations of parts of the boundary of $P$: \[\includegraphics[height=4.476cm]{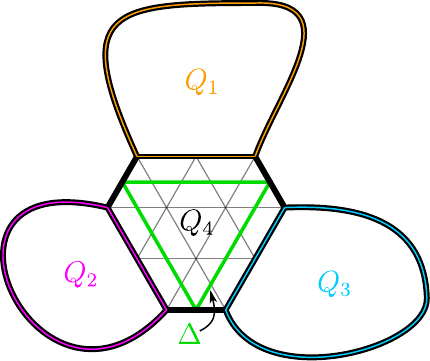}\] The polygon $P$ consists of pieces $Q_1,Q_2,Q_3,Q_4$ as shown: the boundaries of $Q_1,Q_2,Q_3$ are indicated by thin orange, pink, and teal strips, respectively, and $Q_4$ is the closure of $P\setminus(Q_1\cup Q_2\cup Q_3)$. We also allow for each of $Q_1,Q_2,Q_3$ to be a single line segment with area $0$ (a degenerate grid polygon). Because $P$ is primitive, none of $Q_1,Q_2,Q_3$ can have area $1$ or $2$. If $\area(Q_i)\leq 3$ for all $i\in\{1,2,3\}$, then (because it is not a rotation of a polygon in Figure~\ref{FigBill10}) $P$ must be a rotation of one of the polygons \[\includegraphics[height=2.629cm]{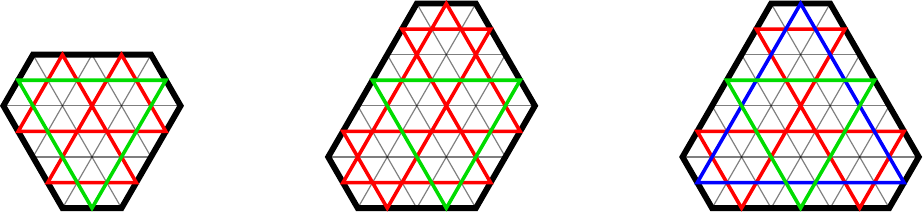},\]
so we can check directly that $\area(P)\geq 6\cyc(P)$. Hence, we may suppose that $\area(Q_i)\geq 4$ for some $i\in\{1,2,3\}$; without loss of generality, assume $\area(Q_1)\geq 4$. Let $Q'=Q_2\cup Q_3\cup Q_4$. Our strategy is to invoke Lemma~\ref{lem:general_cut} with $P_1=Q_1$ and $P_2=Q'$. With these choices of $P_1$ and $P_2$, let $\delta_1$ and $\delta_2$ be as defined in Lemma~\ref{lem:general_cut}. 

Let $t_1$ and $t_2$ be the unique panes in the boundary of $Q_1$ that touch the boundary of $Q'$ but are not contained in $Q'$. Because $P$ is primitive and $Q_1$ has area at least $4$, there are eight possibilities for the orientations of $t_1$ and $t_2$; these possibilities are depicted in Figure~\ref{FigBill11}. 

Suppose first that the orientations of $t_1$ and $t_2$ are as shown in one of the four images on the top of Figure~\ref{FigBill11}. Then $Q_1$ is primitive. If $Q_1$ is not a rotation of one of the grid polygons in Figure~\ref{FigBill10}, then we may apply induction to see that $\area(Q_1)\geq 6\cyc(Q_1)\geq 6(\cyc(Q_1)-\delta_1+1)$. On the other hand, if $Q_1$ is a rotation of one of the polygons in Figure~\ref{FigBill10}, then it must be a rotation of the rightmost polygon in that figure. In this case, $\area(Q_1)=16$, $\cyc(Q_1)=3$, and $\delta_1=2$, so $\area(Q_1)\geq 6(\cyc(Q_1)-\delta_1+1)$ again.

\begin{figure}[ht]
  \begin{center}\includegraphics[width=\linewidth]{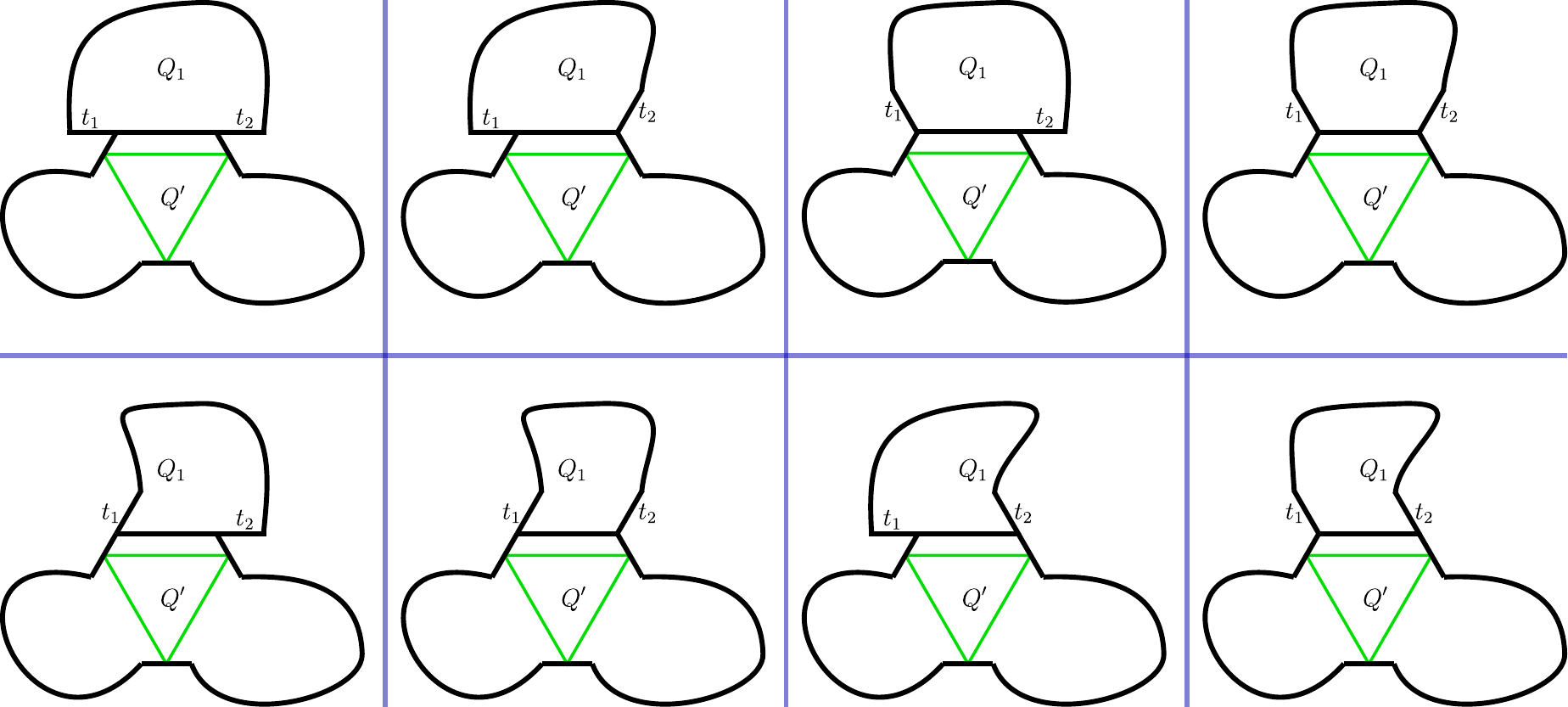}
  \end{center}
  \caption{The eight possible orientations of the panes $t_1$ and $t_2$.}\label{FigBill11}
\end{figure}

Now suppose the orientations of $t_1$ and $t_2$ are as shown in one of the four images on the bottom of Figure~\ref{FigBill11}. In this case, we have $Q_1=\widetilde Q_1\cup T$, where $T$ is a grid polygon that is a triangle of area $1$ and $\widetilde Q_1\cap T$ is a single pane. In the two left (respectively, right) images on the bottom of Figure~\ref{FigBill11}, the triangle $T$ has $t_1$ (respectively, $t_2$) as one of its sides. Note that $\widetilde Q_1$ is primitive. It follows from Corollary~\ref{lem:cut1} that $\cyc(\widetilde Q_1)=\cyc(Q_1)$. If $\widetilde Q_1$ is not a rotation of one of the grid polygons in Figure~\ref{FigBill10}, then we can use induction to see that $\area(\widetilde Q_1)\geq 6\cyc(\widetilde Q_1)=6\cyc(Q_1)\geq 6(\cyc(Q_1)-\delta_1+1)$. On the other hand, if $\widetilde Q_1$ is a rotation of one of the polygons in Figure~\ref{FigBill10}, then (because $\area(\widetilde Q_1)\geq 3$) it is straightforward to check that $\delta_1=2$ and that $\area(Q_1)\geq 6(\cyc(Q_1)-1)$. Thus, $\area(Q_1)\geq 6(\cyc(Q_1)-\delta_1+1)$ in this case as well. 

We have shown that in each of the eight possible cases illustrated in Figure~\ref{FigBill11}, we have 
\begin{equation}\label{Eq:Q1}
\area(Q_1)\geq 6(\cyc(Q_1)-\delta_1+1).    
\end{equation}

Now, $Q'$ is primitive, and it is clearly not a single triangle of area $1$ or a unit hexagon. Invoking Lemma~\ref{lem:general_cut} with $P_1=Q_1$, $P_2=Q'$, and $\eta=2$, we find that $\cyc(P)\leq\cyc(Q_1)+\cyc(Q')-\delta_1-\delta_2+2\leq (\cyc(Q_1)-\delta_1+1)+\cyc(Q')$. If $Q'$ is not a rotation of one of the polygons in Figure~\ref{FigBill10}, then we can use induction to see that $\area(Q')\geq 6\cyc(Q')$. In this case, we can apply \eqref{Eq:Q1} to see that $\area(P)=\area(Q_1)+\area(Q')\geq 6((\cyc(Q_1)-\delta_1+1)+\cyc(Q'))\geq 6\cyc(P)$, as desired. 
On the other hand, if $Q'$ is a rotation of one of the polygons in Figure~\ref{FigBill10}, then it is a rotation of the rightmost such polygon, so $\area(Q')=16$, $\cyc(Q')=3$, and $\delta_2=2$. In this case, $\cyc(P)\leq\cyc(Q_1)+\cyc(Q')-\delta_1-\delta_2+2=\cyc(Q_1)+3-\delta_1$, so invoking \eqref{Eq:Q1} yields $\area(P)=\area(Q_1)+16\geq 6(\cyc(Q_1)-\delta_1+1)+16>6(\cyc(Q_1)+3-\delta_1)\geq 6\cyc(P)$, as desired.
\end{proof}

With the previous proposition out of the way, we can painlessly finish proving Theorems~\ref{thm:area} and~\ref{thm:equality}. Let us first establish one additional piece of terminology. Let $P$ be a grid polygon. It is possible to find a sequence $(P_k)_{k=1}^r$ of grid polygons and a sequence $(Q_k)_{k=1}^{r}$ of primitive grid polygons with $P_1=Q_1$ and $P=P_r$ such that $P_k=P_{k-1}\cup Q_{k}$ and such that $P_{k-1}\cap Q_{k}$ is a single pane for all $k\in\{2,\ldots,r\}$. Moreover, the primitive grid polygons $Q_1,\ldots,Q_{r}$ are uniquely determined up to reordering. We call $Q_1,\ldots,Q_{r}$ the \dfn{primitive pieces} of $P$.

\begin{proof}[Proof of Theorems~\ref{thm:area} and~\ref{thm:equality}]
Let $Q_1,\ldots,Q_r$ be the primitive pieces of $P$. For each $i\in [r]$, Proposition~\ref{prop:indecomposable} tells us that $\area(Q_i)\geq 6\cyc(Q_i)$ if $Q_i$ is not a rotation of one of the grid polygons in Figure~\ref{FigBill10}; if $Q_i$ \emph{is} a rotation of one of the grid polygons in Figure~\ref{FigBill10}, then we can check directly that $\area(Q_i)\geq 6\cyc(Q_i)-6$. It follows from Corollary~\ref{lem:cut1} that $\cyc(P)=\sum_{i=1}^r\cyc(Q_i)-(r-1)$. Thus, \[\area(P)=\sum_{i=1}^r\area(Q_i)\geq \sum_{i=1}^r(6\cyc(Q_i)-6)= 6\left(\sum_{i=1}^r\cyc(Q_i)-(r-1)\right)-6=6\cyc(P)-6.\] This completes the proof of Theorem~\ref{thm:area}. This argument shows that $\area(P)=6\cyc(P)-6$ if and only if $\area(Q_i)=6\cyc(Q_i)-6$ for all $i\in[r]$. By invoking Proposition~\ref{prop:indecomposable} and inspecting the grid polygons in Figure~\ref{FigBill10}, we see that this occurs if and only if the primitive pieces $Q_1,\ldots,Q_r$ are all unit hexagons. This proves Theorem~\ref{thm:equality}.
\end{proof}

\section{Reflections and Next Directions}\label{sec:conclusion}

We believe that this article just scratches the surface of rigid combinatorial billiards systems and their connections with plabic graphs and membranes. In this section, we discuss several variations and potential avenues for future research. 
\subsection{Perimeter vs. Cycles}
Recall Conjecture~\ref{conj:4n-2}, which says that $\perim(P)\geq 4\cyc(P)-2$ for every grid polygon $P$. The grid polygons $P$ satisfying $\perim(P)=4\cyc(P)-2$ seem more sporadic and unpredictable than the equality cases of Theorem~\ref{thm:area}, which are just the trees of unit hexagons by Theorem~\ref{thm:equality}. This gives a heuristic hint as to why Conjecture~\ref{conj:4n-2} is more difficult to prove than Theorem~\ref{thm:area}. 

\begin{figure}[ht]
  \begin{center}\includegraphics[height=3.264cm]{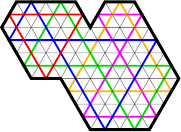}\qquad\qquad\includegraphics[height=3.899cm]{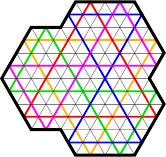}
  \end{center}
  \caption{Two grid polygons with perimeter $18$, each of which has $5$ cycles in its billiards system.}\label{FigBill8}
\end{figure}

\subsection{Other Families of Plabic Graphs} 

Let $G$ be a connected reduced plabic graph with $n$ marked boundary points and $v$ vertices, and let $c$ be the number of cycles in the trip permutation $\pi_G$. Corollary~\ref{cor:plabic} provides inequalities that say how large $n$ and $v$ must be relative to $c$ in the case when $G$ has essential dimension $2$. One can ask for similar inequalities when $G$ is taken from some other interesting family of plabic graphs. One natural candidate for such a family is the collection of plabic graphs of essential dimension $3$; we refer to \cite{LamPostnikov} for further details concerning the definition. It is also natural to consider plabic graphs that can be obtained from polygons in other planar grids besides the triangular grid; Figure~\ref{FigBill12} shows some examples (in these examples, we dismiss our earlier assumption that all vertices in a plabic graph are trivalent). 

\begin{figure}[ht]
  \begin{center}\includegraphics[height=4.5cm]{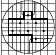}\qquad\includegraphics[height=4.5cm]{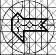}\qquad\includegraphics[height=4.5cm]{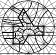}
  \end{center}
  \caption{Plabic graphs obtained from polygons in different planar grids.}\label{FigBill12}
\end{figure}

If $G$ is a plabic graph obtained from the square grid (as on the left of Figure~\ref{FigBill12}), then it is not too difficult to prove that $v\geq 3c-2$ and that $n\geq 4c$; moreover, these bounds are tight. We omit the details. 

\subsection{Regions with Holes} 

Suppose $Q$ is a region in the triangular grid obtained from a grid polygon by cutting out some number of polygonal holes. We can define the billiards system for $Q$ in the same way that we defined it for a grid polygon. It would be interesting to obtain analogues of Theorems~\ref{thm:area}, \ref{thm:perim}, \ref{thm:equality}, and~\ref{thm:HypothesisA} in this more general setting. The resulting analogues of Theorems~\ref{thm:area} and~\ref{thm:perim} might need to incorporate the genus of $Q$. Indeed, Figure~\ref{FigBill9} shows a region $Q$ with genus $1$ for which the inequalities in Theorems~\ref{thm:area} and~\ref{thm:perim} are false as written. 

\begin{figure}[ht]
  \begin{center}\includegraphics[height=2.621cm]{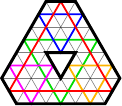}
  \end{center}
  \caption{Trajectories in a triangular grid region of genus $1$. }\label{FigBill9}
\end{figure}

\section*{Acknowledgements}
The first author was supported by the National Science Foundation under Award No.\ DGE--1656466 and Award No.\ 2201907, by a Fannie and John Hertz Foundation Fellowship, and by a Benjamin Peirce Fellowship at Harvard University. The second author was supported by Elchanan Mossel's Vannevar Bush Faculty Fellowship ONR-N00014-20-1-2826 and Elchanan Mossel's Simons Investigator award (622132). We are grateful to Alex Postnikov for helpful conversations. We thank Noah Kravitz for suggesting that we consider billiards systems in square grids and in triangular grid regions with holes cut out, as discussed in Section~\ref{sec:conclusion}. We thank the anonymous referee for helpful suggestions.

\end{document}